\documentclass[12pt]{amsart}

\usepackage{amsmath,amssymb,amsthm,hyperref,cite,color, booktabs, geometry, tikz-cd, mathtools, cleveref, verbatim}

\newtheorem{Theorem}{Theorem}[section]

\newtheorem{Lemma}[Theorem]{Lemma}

\theoremstyle{definition}
\newtheorem{Definition}[Theorem]{Definition}
\newtheorem{Example}[Theorem]{Example}

\newcommand\Res{\mathrm{Res}}
\renewcommand\char{\mathrm{char}}
\newcommand\Br{\mathrm{Br}}
\newcommand\N{\mathrm{N}}
\newcommand\C{\mathrm{C}}
\newcommand\Ind{\mathrm{Ind}}

\renewcommand\ker{\mathrm{ker}}
\newcommand\Hom{\mathrm{Hom}}

\newcommand\soc{\mathrm{soc}}
\renewcommand\top{\mathrm{top}}
\newcommand\vx{\mathrm{vx}}
\renewcommand\mod{\mathrm{mod}}
\newcommand\PSL{\mathrm{PSL}}
\newcommand\SL{\mathrm{SL}}
\newcommand\Syl{\mathrm{Syl}}
\renewcommand\det{\mathrm{det}}

\newcommand\End{\mathrm{End}}
\newcommand\Sl{\mathrm{Sl}}
\newcommand\rad{\mathrm{rad}}
\newcommand\Decom{\mathrm{Decom}}

\title{Slash indecomposability of Endopermutation Scott modules}

\date{}

\author{Lin Wu}
\address{Shenzhen International Center for Mathematics and Department of Mathematics, Southern University of Science and Technology, Shenzhen 518055, China}
\email{12431015@mail.sustech.edu.cn}

\thanks{The author gratefully acknowledges financial support by NSFC (12622101).}

\begin{document}

\maketitle

\begin{abstract}
This paper gives a criterion for an endopermutation Scott module to be slash indecomposable. In particular, under suitable assumptions, we only need to prove certain endopermutation Scott modules over local subgroups are slash indecomposable. 
\end{abstract}

\section{Introduction}

Let $G$ be a finite  group and let $p$ be a prime number. Let $k$ be an algebraically closed field with $\char(k)= p$. The Brauer functor is a fundamental tool for studying p-permutation modules (for background on the Brauer functor and $p$-permutation modules, the reader is referred to \cite[Section 5.4, Section 5.11]{MR3821516}). In 2014, Biland \cite{MR3227305} generalized the notions of Brauer functor and $p$-permutation module to slash functor and Brauer-friendly module, respectively, and then used these new notions in his work \cite{MR3197652} on the gluing problem.

In 2011, Kessar, Kunugi and Mitsuhashi \cite{MR2813564} defined Brauer indecomposability. They gave this definition for the purpose of studying gluing processes, which is a method for studying categorical equivalence between $p$-blocks of finite groups (see \cite{MR2078933, MR2125461}). Brou\'e's gluing method \cite[Theorem 6.3]{MR1308978} gives us a sufficient condition for the Scott module $S(G\times H, \Delta P)$ to induce a stable equivalence of Morita type between two principal blocks. If this is the case, we always have that $S(G\times H, \Delta P)$ is Brauer indecomposable, and no exceptions are known so far. For reference, one could see Koshitani and  Lassueur's works \cite{MR4050079, MR4102107, MR4335855}. So it is important to know whether the Scott module is Brauer indecomposable or not, and this problem has been studied in several papers   \cite{MR3565442, MR2813564, wu2026kernelscottmodulesbrauer}

By \cite{MR3227305}, we know slash functor shares many useful properties with Brauer functors. In 2018, Feng and Li \cite{MR3789018} defined endopermutation Scott module and slash indecomposability. Through these two new notions, they generalized the notions of Scott module and Brauer indecomposability in terms of the slash functor.

There is a criterion for a Scott module to be Brauer indecomposable in \cite[Theorem 1.2]{MR2813564}. 

\begin{Theorem}[Kessar--Kunugi--Mitsuhashi]
Let $P$ be an abelian $p$-subgroup of a finite group $G$. If $\mathcal{F}_P(G)$ is a saturated fusion system then $S(G, P)$ is Brauer indecomposable.
\end{Theorem}

Feng and Li \cite[Theorem 5.5]{MR3789018} generalized this theorem.

\begin{Theorem}[Feng--Li]\label{Th1.2}
Suppose that $P$ is an abelian $p$-subgroup of a finite group $G$ and $M$ is the endopermutation Scott $kG$-module with respect to the fusion stable endopermutation source pair $(P, V)$. If $\mathcal{F}_P(G)$ is a saturated fusion system, then $M$ is slash indecomposable.
\end{Theorem}

In 2017, Ishioka and Kunugi \cite[Theorem 1.3]{MR3565442} gave a new criterion for a Scott module to be Brauer indecomposable in terms of Scott modules over local subgroups.

\begin{Theorem}[Ishioka--Kunugi]
Let $G$ be a finite group and $P$ a $p$-subgroup of $G$. Suppose that $M= S(G, P)$ and that $\mathcal{F}_P(G)$ is saturated. TFAE:

(1) $M$ is Brauer indecomposable.

(2) $\Res_{Q\C_G(Q)}^{\N_G(Q)}S(\N_G(Q), \N_P(Q))$ is indecomposable for each fully normalized subgroup $Q\leq P$.
\end{Theorem}

In this paper, we prove an analogue of \cite[Theorem 1.3]{MR3565442} in terms of the slash functor, which is the following theorem, and this is also the main theorem of this paper. In fact, Watanabe \cite{MR4282221} also considered this problem. We focus on the endopermutation Scott module (see Definition \ref{Def2.8}). Meanwhile, Watanabe considerd a larger class of modules (see \cite[Theorem 5.1]{MR4282221}), but to the best of our knowledge, Watanabe's result does not cover the next theorem directly.

\begin{Theorem}\label{Th3.3}
Let $G$ be a finite group and $P$ a $p$-subgroup of $G$. Let $\mathcal{S}$ be the set of nontrivial fully normalized subgroups of $P$ in the fusion system $\mathcal{F}_P(G)$.  Assume $\mathcal{F}_P(G)$ is saturated. Suppose every fully normalized subgroup of $P$ is a normal subgroup of $P$. Let $(P, V)$ be a fusion-stable endopermutation source pair of $G$ and choose a Frobenius-friendly category $\mathcal{C}_G(P, V)$, a $P$-slash functor $\Sl_{G, P}: \mathcal{C}_G(P, V)\rightarrow \mod(k\N_G(P))$. Let $M$ be an endopermutation Scott module with respect to $(P, V)$ and $\Sl_{G, P}$. For any $Q\in\mathcal{S}$, there exists $t(Q)\in \N_G(P)$ satisfying:

\begin{itemize}
\item $(P, {^{t(Q)}V})$ is a fusion-stable endopermutation source pair of $\N_G(Q)$;
\item Fix a $\mathcal{C}_{\N_G(Q)}(P, {^{t(Q)}V})$, then there exists a $P$-slash functor \begin{equation*}
\Sl_{\N_G(Q), P}: \mathcal{C}_{\N_G(Q)}(P, {^{t(Q)}V})\rightarrow \mod(k\N_G(Q, P))
\end{equation*} 
such that the endopermutation Scott module $M_Q$ with respect to $(P, {^{t(Q)}V})$ and $\Sl_{\N_G(Q), P}$ is isomorphic to a direct summand of $\Res_{\N_G(Q)}^GM$. 
\end{itemize}

For any $Q\in\mathcal{S}$, we fix a quadruple $(t(Q), \mathcal{C}_{\N_G(Q)}(P, {^{t(Q)}V}), \Sl_{\N_G(Q), P}, M_Q)$. If for any $Q\in\mathcal{S}$, $M_Q$ is slash indecomposable, then $M$ is slash indecomposable. 
\end{Theorem}

Theorem \ref{Th1.2} treats the case where $P$ is an abelian subgroup. Although our main theorem does not directly yield slash indecomposability, it makes it possible to determine slash indecomposability by recursively reducing the problem to local $p$-subgroups. In particular, for some finite simple groups, such as $\PSL(2, q)$ and the Tits group $^2F_4(2)'$, we can handle some cases involving non-abelian Sylow subgroups.

Here is the structure of this paper. In section 2, we recall some definitions and properties about slash functor. In section 3, we give the proof of our main theorem \ref{Th3.3}. In section 4, we give two examples involving non-abelian subgroups.

The two examples in Section 4 are originally given by GPT-5.6, and we rewrite it here.

\section{Preliminaries}

For subgroups $Q, R\leq G$, we let
\begin{equation*}
\Hom_G(Q, R)= \{\phi:Q\rightarrow R\mid \text{$\phi$ is induced by a conjugation}\}.
\end{equation*}
For a $p$-subgroup $P\leq G$, the fusion system $\mathcal{F}_P(G)$ is the category whose objects are the subgroups of $P$ and whose morphism set from $Q$ to $R$ is $\Hom_G(Q, R)$. For background on fusion system, one could see \cite{MR2848834}.

Let $V, W$ be two $kG$-modules. We use $V\mid W$ to denote $V$ is isomorphic to a direct summand of $W$.

Now we mainly follow \cite{MR3789018, MR4142080} to recall some basic facts about slash functor. All modules are assumed to be finitely generated.

\begin{Definition}
Let $P$ be a $p$-group and let $V$ be an indecomposable $kP$-module. 

(1) If $P$ is a vertex of $V$, we say $V$ is \emph{capped} and $(P, V)$ is called a \emph{source pair}. 

(2) A source pair $(P, V)$ is called a \emph{fusion-stable endopermutation source pair} of $G$ if for any $p$-subgroup $Q\leq G$ and any two morphisms $\phi_1, \phi_2: Q\rightarrow P$ in $\Hom_G(Q, P)$, the direct sum $\Res_{\phi_1}V\oplus \Res_{\phi_2}V$  is an endopermutation $kQ$-module.

(3) We say two fusion-stable endopermutation source pairs $(P_1, V_1)$ and $(P_2, V_2)$ are \emph{compatible} if for any $p$-subgroup $Q\leq G$ and any two morphisms $\phi_1: Q\rightarrow P_1, \phi_2: Q\rightarrow P_2$, the direct sum $\Res_{\phi_1}V_1\oplus \Res_{\phi_2}V_2$ is an endopermutation $kQ$-module.

(4) We say a $kG$-module $M$ is \emph{Frobenius-friendly} if it is a direct sum of indecomposable $kG$-modules with compatible fusion-stable endopermutation source pairs. We say that two Frobenius-friendly $kG$-modules $L, M$ are \emph{compatible} if the direct sum $L\oplus M$ is also a Frobenius-friendly $kG$-module. 
\end{Definition}

Now we give the definition of Frobenius-friendly category.

\begin{Definition}
Let $\mathcal{C}$ be a full subcategory of the category $\mod(kG)$ consisting of finite-dimensional modules. $\mathcal{C}$ is called \emph{Frobenius-friendly} if: (1) For any object $V$ in $\mathcal{C}$, $V$ is a Frobenius-friendly module; (2) For any two objects $V_1, V_2$ in $\mathcal{C}$, they are compatible Frobenius-friendly modules.
\end{Definition}

A $kG$-module $W$ is called an \emph{endo-$p$-permutation} module if $\End_k(W)$ is a $p$-permutation module (refer to \cite{MR2354860}). By \cite[Remark 3.3]{MR3789018}, we can give an equivalent definition for Frobenius-friendly category in terms of endo-$p$-permutation module.

\begin{Definition}
Let $\mathcal{C}$ be a full subcategory of the category $\mod(kG)$ consisting of finite-dimensional modules. $\mathcal{C}$ is called \emph{Frobenius-friendly} if: (1) For any object $V$ in $\mathcal{C}$, $V$ is an endo-$p$-permutation module; (2) For any two objects $V_1, V_2$ in $\mathcal{C}$, $V_1\oplus V_2$ is also an endo-$p$-permutation module.
\end{Definition}

We shall freely use the two equivalent definitions above without further comment. Now we give the definition of the slash functor.

\begin{Definition}
Let $G$ be a finite group, and $\mathcal{C}$ be a full subcategory of $\mod(kG)$. Let $P$ be a $p$-subgroup of $G$, and $H$ be a subgroup of $G$ such that $P\C_G(P)\leq H\leq \N_G(P)$. What we called a \emph{$P$-slash} functor $\Sl: \mathcal{C}\rightarrow \mod(kH)$ is defined by the following data:

\begin{itemize}
\item for each object $V$ of $\mathcal{C}$, a $kH$-module $\Sl(V)$ s.t. $P\leq \ker(\Sl(V))$.

\item for each pair $V, W$ of objects in the category $\mathcal{C}$, a $k$-linear map
\end{itemize}

\begin{equation*}
\Sl^{V, W}: \Hom_{kP}(V, W)     \rightarrow \Hom_k(\Sl(V), \Sl(W));
\end{equation*}
such that

(1) $\Sl^{V, V}(id)= id$;

(2) $\Sl^{U, W}(v\circ u)= \Sl^{V, W}(v)\circ \Sl^{U, V}(u)$;

(3) for any two objects $V, W$ of $\mathcal{C}$, the map $\Sl^{V, W}$ factors through an isomorphism of $k(\C_G(P)\times \C_G(P))\Delta H$-modules
\begin{equation*}
\Br_{\Delta P}(\Hom_k(V, W))\cong \Hom_k(\Sl(V), \Sl(W)).
\end{equation*}
\end{Definition} 

\textbf{Remark 1.} By (3), $\Sl$ is a $k$-linear functor from $\mathcal{C}$ to $\mod(kH)$.

\textbf{Remark 2.} If $\mathcal{C}$ is the category consisting of $p$-permutation $kG$-modules, then the Brauer functor $\Br_P$ is a $P$-slash functor.

Let $H\leq G$ be a subgroup satisfying $P\C_G(P)\leq H\leq N_G(P)$. Let $L$ be a $kH$-module. Let $\chi: H/ P\C_G(P)\rightarrow k^\times$ be a linear character. We set $\chi L= L$ as a $k$-vector space, and the $H$-action is $h\cdot m= h\chi(\overline{h})m$. 

Now we recall some facts about the slash functor. We will use them in the next section. First, we introduce the existence of the slash functor and some relations between vertices and slash functors. (see \cite[Proposition 2.11]{MR4142080} and \cite[Proposition 3.4, Proposition 3.5]{MR3789018}). 

\begin{Lemma}\label{Lem2.5}
Let $G$ be a finite group, and $\mathcal{C}$ be a Frobenius-friendly category of $kG$-modules. Let $P$ be a $p$-subgroup of the group $G$, and $H$ a subgroup such that $P\C_G(P)\leq H\leq \N_G(P)$.

(1) There exists a $P$-slash functor $\Sl_P: \mathcal{C} \rightarrow \mod(kH)$.

(2) If $\Sl: \mathcal{C} \rightarrow \mod(kH)$ is another $P$-slash functor, then there exists a linear character $\chi: H/ P\C_G(P)\rightarrow k^\times$ and an isomorphism of slash functors $\Sl\cong \chi \Sl_P$.

(3) If $(Q, V)$ is a source pair of an indecomposable direct summand of $\Sl_P(M)$, where $M$ is an object in $\mathcal{C}$, then there is a source pair $(Q', V')$ of an indecomposable direct summand of $M$ such that $P\leq Q\leq Q'$.

(4) Assume $H= \N_G(P)$ and $M$ is an object in $\mathcal{C}$. If $\Sl_P(M)$ admits an indecomposable direct summand with vertex $P$, then the $kG$-module $M$ admits an indecomposable direct summand with vertex $P$.
\end{Lemma}

The following lemma is about transitivity of slash functors (see \cite[Proposition 2.14, Proposition 2.15]{MR4142080}).

\begin{Lemma}\label{Lem2.6}
Let $G$ be a finite group, and $\mathcal{C}$ be a Frobenius-friendly category of $kG$-modules. Let $P\lhd Q$ be $p$-subgroups. Let $\Sl_P: \mathcal{C} \rightarrow \mod(k\N_G(P))$ be a $P$-slash functor. There exists a Frobenius-friendly category $\mathcal{C}'$ of $k\N_G(P)$-modules that contains the essential image of $\Sl_P$. Let $\Sl_Q: \mathcal{C}' \rightarrow \mod(k\N_G(P, Q))$ be a $Q$-slash functor. Then the composition 
\begin{equation}
\Sl_Q\circ \Sl_P: \mathcal{C}\rightarrow \mod(k\N_G(P, Q))
\end{equation}
is a $Q$-slash functor, where $\N_G(P, Q)= \N_G(P)\cap \N_G(Q)$.
\end{Lemma}

The following lemma gives a bijection between indecomposable modules and indecomposable projective modules, which plays a key role in this paper (see \cite[Proposition 2.16]{MR4142080}).

\begin{Lemma}
Let $G$ be a finite group, and let $(P, M)$ be a fusion-stable endopermutation source pair of the group $G$. Let $\mathcal{C}$ be a Frobenius-friendly category of $kG$-modules such that any finite sum of indecomposable $kG$-modules with source pair $(P, M)$ is an object of $\mathcal{C}$. Let $\Sl_P: \mathcal{C}\rightarrow \mod(k\N_G(P))$ be a $P$-slash functor. Then the mapping $X\mapsto \Sl_P(X)$ induces a one-to-one correspondence between the isomorphism classes of indecomposable $kG$-modules with source pair $(P, M)$ and the isomorphism classes of projective indecomposable $k\N_G(P)/ P$-modules.
\end{Lemma}

By the above Lemma, we can give the definitions of endopermutation Scott module and slash indecomposability (see \cite[Definition 4.1, Definition 5.1]{MR3789018}).

\begin{Definition}\label{Def2.8}[Feng--Li]
Let $(P, V)$ be a fusion-stable endopermutation source pair, and let $\mathcal{C}, \Sl_P$ be as in the above lemma. 

(1)There is a unique (up to isomorphism) indecomposable module $M$ with source pair $(P, V)$ in $\mathcal{C}$ such that the $k[\N_G(P)/ P]$-module $\Sl_P(M)$ has trivial top, i.e. $\Sl_P(M)/ \rad(\Sl_P(M))\cong {_{\N_G(P)/P}k}$, which is a 1-dimensional module with trivial $\N_G(P)/ P$-action. Then $M$ is called an \emph{endopermutation Scott module} with respect to $(P, V)$ and $\Sl_P$.

(2) A Frobenius-friendly $kG$-module $V$ is said to be \emph{slash indecomposable} if for any $p$-subgroup $Q\leq G$ and $Q$-slash functor $\Sl_Q: \mathcal{C}\rightarrow \mod(k\N_G(Q))$, where $\mathcal{C}$ is a Frobenius-friendly category containing $V$, we have $\Sl_Q(V)$ is indecomposable or zero as a $k[Q\C_G(Q)/ Q]$-module.
\end{Definition}

\textbf{Remark 1.} For any linear character $\chi$ of $\N_G(Q)/ Q\C_G(Q)$, we have $\chi(\overline{x})= 1$ if $x\in \C_G(Q)$. Thus slash indecomposability is independent of the choice of slash functors (see \cite[Remark 5.2]{MR3789018}).

\textbf{Remark 2.} By \cite[Lemma 2.1]{MR3565442}, The Scott module $S(G, P)$ is the endopermutation Scott module with respect to $(P, {_Pk})$ and $\Br_P$. 

\textbf{Remark 3.} Suppose that $M$ is an indecomposable Frobenius-friendly $kG$-module and $P$ is a $p$-subgroup of $G$. Then $\Sl_P(M)\neq 0\Leftrightarrow P\leq_G vx(M)$ (see \cite[Remark 3.8]{MR3789018}).

The following fact will be frequently used, and it follows from \cite[Lemma 4.7]{MR3789018}.

\begin{Lemma}\label{Lem2.9}
Suppose that $R, H$ are subgroups of a finite group $G$ such that $R\leq H$. Assume $R$ is a $p$-subgroup. Let $\mathcal{C}$ be a Frobenius-friendly category of $kG$-modules and $V$ is an object in $\mathcal{C}$. Let $\Sl_{G, R}: \mathcal{C}\rightarrow \mod(k\N_G(R))$ be an $R$-slash functor over $G$. Then there exists an $R$-slash functor over $H$, denoted by $\Sl_{H, R}$, such that $\Res_{\N_H(R)}^{\N_G(R)}(\Sl_{G, R}(V))= \Sl_{H, R}(\Res_H^G(V))$. 
\end{Lemma}

\section{Proof of Theorems}

Let $G$ be a finite group and let $P$ be a $p$-subgroup. Let $(P, V)$ be a fusion-stable endopermutation source pair of $G$. We use $\mathcal{C}_G(P, V)$ to denote a Frobenius-friendly category of $kG$-modules such that any finite sum of indecomposable $kG$-modules with source pair $(P, V)$ is an object of $\mathcal{C}_G(P, V)$ (There may be different choices for $\mathcal{C}_G(P, V)$, using this symbol indicates that we choose one of them).

Let $M$ be a finite-dimensional $kG$-module. We use $\Decom(M)$ to denote the set of all indecomposable summands appearing in some indecomposable decomposition of $M$ (There may be different choices for $\Decom(M)$, but by Krull--Schmidt--Azumaya theorem we know it is unique up to isomorphism).

\begin{Lemma}\label{lem1}
Let $G$ be a finite group and let $P$ be a $p$-subgroup. Let $H$ be a subgroup of $G$ such that $P\leq H$. Let $(P, V)$ be a fusion-stable endopermutation source pair of $G$ and we choose a $\mathcal{C}_G(P, V)$. Let $\Sl_{G, P}: \mathcal{C}_G(P, V)\rightarrow \mod(k\N_G(P))$ be a $P$-slash functor. Let $M$ be an endopermutation Scott module with respect to $(P, V)$ and $\Sl_{G, P}$. Then there exists a fusion-stable endopermutation source pair $(P, {^tV})$ of $H$, where $t\in \N_G(P)$, such that there is an endopermutation Scott module $\widetilde{M}$ with source pair $(P, {^tV})$ satisfying $\widetilde{M}\mid \Res_H^G(M)$.
\end{Lemma}
\begin{proof}
Let $\mathcal{C}_2$ be a full subcategory of $\mod(kH)$ consisting of $\Res_H^G(M)$ and $\Decom(\Res_H^G(M))$. $\mathcal{C}_2$ is a Frobenius-friendly category and there exists a Slash functor $\Sl_{H, P}: \mathcal{C}_2\rightarrow \mod(k\N_H(P))$ such that $\Sl_{H, P}(\Res_H^G(M))= \Res_{\N_H(P)}^{\N_G(P)}\Sl_{G, P}(M)$ by Lemma \ref{Lem2.9}. As a $k[\N_G(P)/ P]$-module, $\Sl_{G, P}(M)$ is an indecomposable projective module satisfying
\begin{equation*}
\top(\Sl_{G, P}(M))\cong \soc(\Sl_{G, P}(M))\cong {_{\N_G(P)/ P}k}.
\end{equation*}
So there exists an indecomposable direct summand $X\mid \Res_H^G(M)$ such that 
\begin{equation*}
\Hom_{k\N_H(P)}(k, \Sl_{H, P}(X))\neq 0.
\end{equation*}
Thus $P\leq_H \vx(X)\leq_G P$. So there exists $t\in \N_G(P)$ such that $X$ has a source pair $(P, {^tV})$. We choose a $\mathcal{C}_H(P, {^tV})$. Let $\Sl_{H, P}': \mathcal{C}_H(P, {^tV})\rightarrow \mod(k\N_H(P))$ be a $P$-slash functor. By Lemma \ref{Lem2.5} (2), after twisting by a linear character we may assume $\Sl_{H, P}'(X)= \Sl_{H, P}(X)$. It is known that $\Sl_{H, P}'(X)$ is an indecomposable projective $k[\N_H(P)/ P]$-module. Since $\Sl_{H, P}'(X)$ contains a trivial submodule $_{\N_H(P)/ P}k$, we have $\top(\Sl_{H, P}'(X))= _{\N_H(P)/ P}k$. Thus $X$ is an endopermutation Scott module with respect to the source pair $(P, {^tV})$ and slash functor $\Sl_{H, P}'$.
\end{proof}

\begin{Lemma}\label{lem2}
Let $G$ be a finite group and let $P$ be a $p$-subgroup. Suppose $Q$ is a $p$-subgroup of $P$. Let $x\in G$. Let $(P, S)$ be a fusion-stable endopermutation source pair of $G$ and we choose a $\mathcal{C}_G(P, S)$. Let $\Sl_{G, Q}: \mathcal{C}_G(P, S)\rightarrow \mod(k\N_G(Q))$ be a $Q$-slash functor. Set $(^x\Sl_{G, Q})(Y)= {^x\Sl_{G, Q}(Y)}$ for any object $Y$ of $\mathcal{C}_G(P, S)$. Then $^x\Sl_{G, Q}: \mathcal{C}_G(P, S)\rightarrow \mod(k\N_G(^xQ))$ is a $^xQ$-slash functor. 
\end{Lemma}
\begin{proof}
Let $V, W$ be any two objects of $\mathcal{C}_G(P, S)$. By the definition of $Q$-slash functor, there exists a commutative diagram of $k(C_G(^xQ)\times C_G(^xQ))\Delta\N_G(^xQ)$-modules:

\begin{center}
\begin{tikzcd}
{\Hom_{k {^xQ}}({^xV}, {^xW})} \arrow[rd] \arrow[rr] &                                                             & {\Hom_k({^x\Sl_{G, Q}(V)}, {^x\Sl_{G, Q}(W)})} \\
                                                    & {\Br_{\Delta(^xQ)}(\Hom_k({^xV}, {^xW}))} \arrow[ru, "\cong"] &                                             .
\end{tikzcd}
\end{center}

Since $V\cong {^xV}$ as $kG$-modules, we get the following commutative diagram of $k\N_{G\times G}(\Delta(^xQ))$-modules:

\begin{center}
\begin{tikzcd}
{\Hom_{k {^xQ}}(V, W)} \arrow[r, "\cong"] \arrow[d]  & {\Hom_{k{^xQ}}({^xV}, {^xW})} \arrow[d]  \\
{\Br_{\Delta(^xQ)}(\Hom_k(V, W))} \arrow[r, "\cong"] & {\Br_{\Delta(^xQ)}(\Hom_k({^xV}, {^xW}))}.
\end{tikzcd}
\end{center}

We define $\Sl_{G, {^xQ}}(V)= {^x\Sl_{G, Q}(V)}$. Then we get the following commutative diagram of $k(C_G(^xQ)\times C_G(^xQ))\Delta\N_G(^xQ)$-modules:

\begin{center}
\begin{tikzcd}
{\Hom_{k {^xQ}}(V, W)} \arrow[rd] \arrow[rr] &                                                             & {\Hom_k(\Sl_{G, {^xQ}}(V), \Sl_{G, {^xQ}}(W))} \\
                                                    & {\Br_{\Delta(^xQ)}(\Hom_k(V, W))} \arrow[ru, "\cong"] &.                                      
\end{tikzcd}
\end{center}

So $\Sl_{G, {^xQ}}$ is a $^xQ$-slash functor. 
\end{proof}

\begin{proof}[Proof of Theorem \ref{Th3.3}]
First, the existence of $t(Q)$ follows from Lemma \ref{lem1}. Now we assume for any $Q\in \mathcal{S}$, $M_Q$ is slash indecomposable.

Claim: For any subgroup $Q\leq P$ and any $Q$-slash functor $\Sl_{G, Q}: \mathcal{C}_G(P, V)\rightarrow \mod(k\N_G(Q))$, $\Res_{Q\C_G(Q)}^{\N_G(Q)}\Sl_{G, Q}(M)$ is indecomposable.

Proof of claim: First we assume $Q= 1$. Then by Lemma \ref{Lem2.5} (2), $\Sl_{G, Q}$ is the identity functor. Therefore, $\Res_{Q\C_G(Q)}^{\N_G(Q)}\Sl_{G, Q}(M)= M$, which is indecomposable. Now we assume $Q> 1$ and we use induction on $|P|/ |Q|$. 

Case $P= Q$: By \cite[Lemma 4.8]{MR3789018}.

Case $P> Q$: First we assume $Q\in\mathcal{S}$. There exists a $k\N_G(Q)$-module $N_Q$ such that $\Res_{\N_G(Q)}^G(M)\cong M_Q\oplus N_Q$. Let $\mathcal{C}$ be a full subcategory of $\mod(k\N_G(Q))$ consisting of $\Res_{\N_G(Q)}^G(M), M_Q, N_Q$ and $\Decom(M_Q), \Decom(N_Q)$. Then $\mathcal{C}$ is a Frobenius-friendly category. There exists a $Q$-slash functor $\widetilde{\Sl}_{N_G(Q), Q}: \mathcal{C}\rightarrow \mod(k\N_G(Q))$ such that $\widetilde{\Sl}_{\N_G(Q), Q}(\Res_{\N_G(Q)}^G(M))= \Sl_{G, Q}(M)$. So we get
\begin{equation*}
\Res_{Q\C_G(Q)}^{\N_G(Q)}(\Sl_{G, Q}(M))\cong \Res_{Q\C_G(Q)}^{\N_G(Q)}(\widetilde{\Sl}_{\N_G(Q), Q}(M_Q))\oplus \Res_{Q\C_G(Q)}^{\N_G(Q)}(\widetilde{\Sl}_{\N_G(Q), Q}(N_Q)).
\end{equation*}
Since $M_Q$ is slash indecomposable, we only need to prove $\widetilde{\Sl}_{\N_G(Q), Q}(N_Q)= 0$. 

If $\widetilde{\Sl}_{\N_G(Q), Q}(N_Q)\neq 0$, it has an indecomposable summand $L$. Let $R$ be a vertex of $L$. By Lemma \ref{Lem2.5}, there exists an indecomposable summand $\widetilde{L}\mid N_Q$ with vertex $B$ such that $B\geq R\geq Q$. If $R= Q$, then $P= Q$ by Lemma \ref{Lem2.5} (4), which contradicts our assumption $P> Q$. Since $\widetilde{L}\mid \Res_{\N_G(Q)}^G\Ind_P^G(V)$, by Mackey's decomposition formula, we get $\widetilde{L}\mid \Ind_{\N_{^tP}(Q)}^{\N_G(Q)}\Res_{\N_{^tP}(Q)}^{^tP}(^tV)$ for some $t\in G$. So $Q\leq B\leq_{\N_G(Q)} \N_{^tP}(Q)$. By \cite[Lemma 3.2]{MR3565442}, we have $\N_{^tP}(Q)\leq_{\N_G(Q)} P$. Thus we may assume $Q\leq R\leq B\leq P$.

By Lemma \ref{Lem2.6}, there exists a Frobenius-friendly category $\mathcal{D}$ of $k\N_G(Q)$-modules that contains the essential image of $\Sl_{G, Q}$. Let $\Sl_{N_G(Q), R}': \mathcal{D}\rightarrow \mod(k\N_G(Q, R))$ be an $R$-slash functor. By Lemma \ref{Lem2.6}, we know $\Sl_{N_G(Q), R}'\circ \Sl_{G, Q}: \mathcal{C}_G(P, V)\rightarrow \mod(k\N_G(Q, R))$ is an $R$-slash functor. Since $\Res_{\N_G(Q, R)}^{\N_G(R)}\circ \Sl_{G, R}$ is also an $R$-slash functor, we may twist $\Sl_{\N_G(Q), R}'$ by a linear character $\chi: \N_G(Q, R)/ R\C_G(R)\rightarrow k^\times$ such that 
\begin{equation*}
\Sl_{\N_G(Q), R}'\circ \Sl_{G, Q}(M)\cong \Res_{\N_G(Q, R)}^{\N_G(R)}(\Sl_{G, R}(M)).
\end{equation*}
By the induction hypothesis, $\Sl_{\N_G(Q), R}'\circ \Sl_{G, Q}(M)$ is indecomposable. Now we let $\mathcal{D}'$ be a full subcategory of $\mod(k\N_G(Q))$ consisting of $\Sl_{G, Q}(M), \widetilde{\Sl}_{\N_G(Q), Q}(M_Q), \widetilde{\Sl}_{\N_G(Q), Q}(N_Q)$ and $\Decom(\widetilde{\Sl}_{\N_G(Q), Q}(M_Q)), \Decom(\widetilde{\Sl}_{\N_G(Q), Q}(N_Q))$. We require $L\in \Decom(\widetilde{\Sl}_{\N_G(Q), Q}(N_Q))$. There exists an $R$-slash functor $\Sl_{\N_G(Q), R}'': \mathcal{D}'\rightarrow \mod(k\N_G(Q, R))$ such that 
\begin{equation*}
\Sl_{\N_G(Q), R}'\circ \Sl_{G, Q}(M)= \Sl_{\N_G(Q), R}''\circ \Sl_{G, Q}(M).
\end{equation*}

Since $\Sl_{\N_G(Q), R}''(L)\neq 0$, we get $\Sl_{\N_G(Q), R}''(\widetilde{\Sl}_{\N_G(Q), Q}(N_Q))\neq 0$. For any $P$-slash functor $\Sl_P: \mathcal{D}'\rightarrow \mod(k\N_G(Q, P))$, $\Sl_P\circ\widetilde{\Sl}_{\N_G(Q), Q}(M_Q)\neq 0$ since $P$ is a vertex of $M_Q$. Thus there exists an indecomposable summand $Y\mid \widetilde{\Sl}_{\N_G(Q), Q}(M_Q)$ such that $\Sl_P(Y)\neq 0$. Let $T$ be a vertex of $Y$. Then by Lemma \ref{Lem2.5}, we have 
\begin{equation*}
P\leq_{\N_G(Q)} T\leq_{\N_G(Q)} P,
\end{equation*}
where the first inclusion comes from $\Sl_P(Y)\neq 0$. So we may assume $T= P$. In order to prove $\Sl_{\N_G(Q), R}''(\widetilde{\Sl}_{\N_G(Q), Q}(M_Q))\neq 0$, we only need to prove $R\leq_{\N_G(Q)} P$, but we already have $R\leq P$.

Combining  $\Sl_{\N_G(Q), R}''(\widetilde{\Sl}_{\N_G(Q), Q}(N_Q))\neq 0$ and $\Sl_{\N_G(Q), R}''(\widetilde{\Sl}_{\N_G(Q), Q}(M_Q))\neq 0$, we know $\Sl_{\N_G(Q), R}''\circ \Sl_{G, Q}(M)$ is a direct sum of two nonzero modules, which contradicts its indecomposability.

Then we assume $Q$ is just a subgroup of $P$. There exists $x\in G$ such that $^xQ\leq P$ is fully normalized. By Lemma \ref{lem2}, $^x\Sl_{G, Q}: \mathcal{C}_G(P, V)\rightarrow \mod(k\N_G(^xQ))$ is a $^xQ$-slash functor. So $\Res_{^xQ\C_G(^xQ)}^{\N_G(^xQ)}{^x\Sl_{G, Q}}(M)$ is indecomposable. Thus $\Res_{Q\C_G(Q)}^{\N_G(Q)}\Sl_{G, Q}(M)$ is indecomposable.

By the definition of slash indecomposability we also need to consider the case $Q\nleq P$, but in such case, we only need to use Lemma \ref{lem2} and the Remark 3 after Definition \ref{Def2.8} directly.

\end{proof}

\section{Examples}
A finite group is called a \emph{Dedekind group} if all subgroups are normal. A nonabelian Dedekind group is called a \emph{Hamiltonian group}. We will use the following fact: A Dedekind $p$-group is abelian when $p$ is odd, which is a corollary of the following theorem \cite[p.325, Theorem 7.12]{MR4928289}.

\begin{Theorem}[Dedekind]
Assume $G$ is nonabelian and every subgroup of $G$ is normal in $G$. Then $G= Q\times A\times B$, where $Q$ is a quaternion group of order $8$, $A$ is an abelian group of odd order and $B$ is an abelian group of exponent $1$ or $2$. Conversely, if $G$ has the given structure, then every subgroup of $G$ is normal in $G$.
\end{Theorem}

In this section, we give two examples. In particular, there is an infinite family of finite simple groups such that for any group $G$ in this family, there exists a non-Dedekind Sylow subgroup $P$ satisfying the group-theoretic hypotheses in Theorem \ref{Th3.3}. 

\begin{Example}
Let $q$ be an odd prime power. Assume $q\equiv 7~\text{or}~9(\text{mod}~16)$. Let $G= \PSL(2, q)~\text{or}~\SL(2, q)$, and let $P\in \Syl_2(G)$, then 

(1) $\mathcal{F}_P(G)$ is saturated; 

(2) If $Q\leq P$ is a fully normalized subgroup, then $Q\lhd P$;

(3) $P$ is not Dedekind. 
\end{Example}
\begin{proof}
By \cite[p.8, Theorem 2.3]{MR2848834}, $\mathcal{F}_P(G)$ is saturated.

First, we assume $G= \PSL(2, q)$. Since $q\equiv 7~\text{or}~9(\text{mod}~16)$, we get $v_2(q^2-1)= 4$. Note that $|G|= q(q^2-1)/ 2$. So $|P|= 8$. By \cite[p.223, Theorem 8.27]{MR4928289}, we get $P\cong D_8$. Let $Q$ be a subgroup of $P$ such that $Q\ntriangleleft P$. By direct calculation, we know $|Q|= 2$ and $|\N_P(Q)|= 4$. By \cite[p.297, paragraph 4]{MR3418059}, we know all involutions in $G$ are conjugate with each other. So there exists $g\in G$ such that $gQg^{-1}= Z(P)$. But $|\N_P(Q)|= 4< 8= |\N_P(Z(P))|= |P|$. So $Q$ is not fully normalized.

Now we assume $G= \SL(2, q)$. $v_2(|G|)= v_2(q(q^2- 1))= 4$. So $|P|= 16$. Let $\pi: G\rightarrow \PSL(2, q)$ be the canonical epimorphism. Let $x$ be an involution in $G$, then $x^2= I$. Thus $x$ is diagonalizable and the set of eigenvalues of $x$ can only be $\{1\}$ or $\{-1\}$ since $\det(x)= 1$. So $x= -I$ and $\ker(\pi)= \langle -I\rangle$. Therefore, $\pi(P)\cong D_8$. $P$ is a group of order 16 and has a unique involution. The quotient of $P$ by this involution is $D_8$. By direct calculation, we know 
\begin{equation*}
P= \langle a, b\mid a^8= 1, b^2= a^4, bab^{-1}= a^{-1}\rangle\cong Q_{16}. 
\end{equation*}

Let $Q$ be a subgroup of $P$ such that $Q\ntriangleleft P$. By direct calculation, $|Q|= 4$ and $|\N_P(Q)|= 8$. Since $Q$ contains the unique involution $-I$, we get $|\pi(Q)|= 2$ and thus there exists $g\in G$ such that $\overline{g}\pi(Q)\overline{g}^{-1}= Z(\pi(P))= \pi(\langle a^2\rangle)$. So $gQg^{-1}= \langle a^2\rangle$. Now we have $|\N_P(Q)|= 8< 16= |\N_P(\langle a^2\rangle)|$. So $Q$ is not fully normalized.

$D_8, Q_{16}$ are non-abelian and non-Dedekind follows from direct calculation.

In order to prove there are infinitely many such groups, it suffices to prove there are infinitely many odd primes $q$ satisfying $q\equiv 7~\text{or}~9(\text{mod}~16)$, which follows from the well-known Dirichlet's theorem.
\end{proof}

\begin{Lemma}\label{lemnor}
Let $G$ be a finite group and let $P\in \Syl_p(G)$ (we do not require $p$ is odd). TFAE: (1) For any fully normalized subgroup $Q$ in $\mathcal{F}_P(G)$, we have $Q\lhd P$; (2) For any $Q\leq P$, we have $p\nmid [G: \N_G(Q)]$.
\end{Lemma}
\begin{proof}
$(1)\Rightarrow (2):$ Let $Q\leq P$ and let $g\in G$ be such that $R:= {^gQ}$ is a fully normalized subgroup of $P$. By our assumption, $R\lhd P$ and thus $Q= {^{g^{-1}}R}\lhd {^{g^{-1}}P}$. So $p\nmid [G: \N_G(Q)]$.

$(2)\Rightarrow (1):$ Suppose $Q\leq P$ is fully normalized. There exists an $S\in \Syl_p(G)$ such that $Q\leq S\leq \N_G(Q)$. There exists $g\in G$ such that $^gS= P$. So 
\begin{equation*}
|\N_P({gQg^{-1}})|= |g\N_{g^{-1}Pg}(Q)g^{-1}|= |\N_S(Q)|= |S|= |P|.
\end{equation*}
Since $|\N_P({gQg^{-1}})|\leq |\N_P(Q)|$, we get $Q\lhd P$.
\end{proof}

Now we give an example considering Sylow $3$-subgroup. 

\begin{Example}
Let $G= {^2F_4(2)'}$, the Tits group. Let $P\in \Syl_3(G)$.  Then

(1) $\mathcal{F}_P(G)$ is saturated; 

(2) If $Q\leq P$ is a fully normalized subgroup, then $Q\lhd P$;

(3) $P$ is not Dedekind.
\end{Example}
\begin{proof}
By ATLAS \cite{MR827219}, $|G|= 2^{11}\cdot3^3\cdot5^2\cdot13$ and $\PSL(3, 3)$ is a subgroup containing a Sylow $3$-subgroup of order $27$. Let $P$ be the subgroup of $\PSL(3, 3)$ consisting of upper unitriangular matrices. $P$ is not abelian. Let $Q\leq P$. We only need to show $3\nmid[G: \N_G(Q)]$. If $|Q|\in \{1, 27\}$, then we are done. If $|Q|= 9$, then $[P: Q]= 3$, which implies $Q\lhd P$. So $P\leq \N_G(Q)$ and $3\nmid [G: \N_G(Q)]$. 

Assume $|Q|= 3$. Let $x\in Q\backslash\{1\}$. Again, by ATLAS, there exists a unique conjugacy class containing elements of order $3$ and the size of the centralizer is $108= 2^2\cdot 3^3$. So $3\nmid [G: \N_G(Q)]$.
\end{proof}

\section*{Acknowledgement}

The author thanks Zhicheng Feng for his careful reading of the manuscript and his valuable comments.

\bibliography{reference}
\bibliographystyle{abbrv}

\end{document}